\documentclass[12pt]{amsart}

\usepackage{amsmath,amssymb,mathrsfs}
\usepackage{a4wide}
\usepackage{amsthm} 

\usepackage{srcltx} 

 \usepackage[dvips]{graphicx,color}
\usepackage{amsopn,cite,mathrsfs}
\usepackage{amsfonts}
\usepackage{eucal}

\usepackage{times}

\numberwithin{equation}{section}

\newtheorem{thm}{Theorem}[section]
\newtheorem{lemma}[thm]{Lemma}

\newtheorem*{Leucking}{Luecking's Theorem }
\newtheorem*{cor1prime}{Corollary 1' }

\newtheorem*{mthm}{Main Theorem}
\newtheorem*{proof-mthm}{Proof of Main Theorem}
\newtheorem*{proofcor2}{Proof of Corollary 2}
\newcommand{\C}{{\mathbb C}}
\newcommand{\E}{{\mathbb E}}

\newcommand{\Castinfty}{{{{\mathbb C}}^\infty}}
\newcommand{\A}{{\mathcal A}(\S)}

\newcommand{\T}{{\mathbb T}}

\newcommand{\D}{{\mathbb D}}

 \newcommand{\blem}{\begin{lemma}}
 \newcommand{\elem}{\end{lemma}}
\newcommand{\bpve}{\begin{proof}}
\newcommand{\epve}{\end{proof}}

\def\B{\mathbb{B}}
\def\C{\mathbb{C}}
\def\D{\mathbb{D}}
\def\H{\mathbb{H}}
\def\cD{\overline{\mathbb{D}}}
\def\N{\mathbb{N}}

\def\S{\mathbb{S}}
\def\Sde{\mathbb{S} \times \mathbb{S} }

\def\Ptau{\mathcal{P_\tau}}

\usepackage{amssymb}
\usepackage{amsmath}
\usepackage{amsfonts}
\usepackage[colorlinks=true,linkcolor=red]{hyperref}
\newtheorem{theorem}{Theorem}
\theoremstyle{plain}
\newtheorem{corollary}{Corollary}
\newtheorem{definition}{Definition}

\newtheorem{proposition}{Proposition}
\newtheorem{remark}{Remark}
\numberwithin{equation}{section}

\begin{document}

\title{  A Covariance Equation}
\author{El Hassan Youssfi}
\email[E.H. Youssfi]{el-hassan.youssfi@univ-amu.fr}
 
\thanks{In memory of Peter Maserick}

\begin{abstract} Let $\S$ be a  commutative semigroup  with identity $e$ and  let $\Gamma$ be a compact subset  in the pointwise convergence topology of the space $\S'$ of all non-zero multiplicative functions on $\S.$  Given a continuous  function $F: \Gamma \to \C$ and a  complex regular Borel measure $\mu$ on $\Gamma$ such that $\mu(\Gamma) \not = 0.$  It is shown that    $$ \mu(\Gamma) \int_{\Gamma} \varrho(s)  \overline{\varrho(t)} |F|^2(\varrho) d\mu(\varrho) =  \int_{\Gamma} \varrho(s) F(\varrho) d\mu(\varrho) \int_{\Gamma}  \overline{\varrho(t) F(\varrho)} d\mu(\varrho) $$
for all $(s, t) \in \S\times \S$ if and only if for some $\gamma \in \Gamma, $  the support of 
 $\mu$ is contained is contained in $\{ F  = 0 \} \cup \{\gamma\}$.   Several applications of this characterization are derived.  In particular, the reduction of our theorem to the semigroup of non-negative integers $(\N_{0}, +)$ solves a problem  posed  by El Fallah, Klaja, Kellay, Mashregui and Ransford in a more general context.
 More consequences of our results are given, some of them illustrate the probabilistic flavor behind the problem studied herein  and others establish extremal properties of analytic kernels.   
\end{abstract}

\maketitle
\section{Introduction and statement of the main results} Let $\S$ be a commutative semigroup with identity $e$ and let
$\S'$ be the completely regular space consisting of all non-zero multiplicative complex valued functions on $\S$ furnished  with the pointwise convergence topology.
In this paper, we characterize the compactly supported complex regular Borel measures  $\mu$ on $\S'$   that satisfy  the following  equation
\begin{eqnarray*} \mu(\S') \int_{\S'} \varrho(s) \overline{\varrho(t)} |F(\varrho)|^2d\mu(\varrho) = \int_{\S'} \varrho(s) F(\varrho) d\mu(\varrho) \int_{\S'}  \overline{\varrho(t) F(\varrho)}d\mu(\varrho)  \end{eqnarray*}

for all $s, t \in \S$, where $F$ is a fixed continuous function on $\S'.$  If $ \mu(\S') =1,$ then  setting $\chi_s (\varrho):= \varrho(s)$, the latter equality can be written in the {\it covariance equation} form
\begin{eqnarray*} \int_{\S'}\left( F\chi_s  - \int_{\S'}F\chi_s d\mu\right) 
\left( \overline{F\chi_t}  - \int_{\S'}\overline{F\chi_t} d\mu\right) d\mu = 0
  \end{eqnarray*}
for all $s, t \in \S$.

   To handle this problem we  first use the action of the algebra of shift operators on functions on $\Sde.$ Then  we appeal to Luecking's theorem on finite rank Toeplitz operators \cite{L}.
    Several  applications of our solution will be given in Sections 5 and 6.  They all offer new results. The reduction of our charaterization to the additive semigroup $\N_{0}$ of all non-negative integers  solves a problem which was left open in recent paper by El Fallah, Kallaj, Kellay, Mashregui and Ransford in a more general context.  Indeed, we provide a solution to this question in the multi-dimensional setting.  
     
     Let ${\S}$   denote a  multiplicative commutative  semigroup  with
identity $e$. A function $\varrho : \S \to \C$ is said to be  multiplicative if its satisfies
$\varrho(st) = \varrho(s)\varrho(t)$ for all $(s,t) \in \S^2$. It is clear, that if $\varrho$ is non-zero  multiplicative on $\S$, then $\varrho(e) = 1.$ We denote by $\S'$ the set of all non-zero multiplicative complex valued functions on $\S$ and equip it  with pointwise convergence topology under which it is a completely regular topological space.

 Let $\mu$ be a compactly supported complex regular Borel measure $\mu$ on  $\S'$. It is known (see  Rudin  \cite{Ru})  that $\mu$ has a  polar decomposition 
 $\mu = h|\mu|$,  where $|\mu|$ is a nonnegative measure, the total variation of $\mu$, on $\S^{\ast}$ and $h = h_{\mu} : \S' 
\rightarrow \C$ a Borel measurable function with $|h| \equiv 1$.

The  generalized Laplace transform of 
$\mu$ is the function defined on $\Sde$ by
\begin{eqnarray}\label{laplace}\widehat{\mu}(s, t) := \int_{ \S'}\varrho(s) \overline{\varrho(t)}\, d\mu(\varrho), \ (s,t) \in \Sde. \end{eqnarray}and the variation of  $\widehat{\mu}$ is the function
$$ \widehat{|\mu|}(s, t) := \int_{ \S'}\varrho(s)\overline{\varrho(t)} \, d|\mu|(\varrho), \ (s,t) \in \Sde. $$

 The above generalized Lapalace transforms can be viewed as $BV$-functions in the sense of Maserick \cite{M2} on the semigroup on $ \Sde$ with appropriate involution.

If $F$ is a continuous  function on $\S'$, we denote by $F\mu$ the measure with density $F$ with respect to $\mu.$
\begin{definition}
Given complex Borel measure $\mu$ on a compact subset $\Gamma$ of $\S'$ and a continuous function $F: \Gamma \to \C$, we shall say that the measure $\mu$ satisfies {\it the covariance equation with respect to $F$}  if

\begin{eqnarray}\label{main-equality} \mu(\Gamma) \widehat{ |F|^2\mu}(s, t) =  \widehat{ F\mu}(s, e)  \widehat{ \overline{F}\mu}(e, t)\end{eqnarray}
for all $s, t \in \S$.
\end{definition}

It is obvious that if  $\mu\left(\Gamma\right) = 0, $  then $\mu$ satisfies (\ref{main-equality}) if and only if either
 \begin{eqnarray}\label{main-equality1}   \int_{\Gamma} \varrho(s) F(\varrho) d\mu(\varrho) =  0 \ \ \forall s\in \S  \end{eqnarray}
 or 
 \begin{eqnarray}\label{main-equality2}    \int_{\Gamma}  \overline{\varrho(s)F(\varrho) }d\mu(\varrho) = 0, \ \ \forall s\in \S.
 \end{eqnarray}
 Our goal herein is  to describe the measures $ \mu$ which satisfy     (\ref{main-equality}).  The 
 main result is the following:
\begin{mthm}\label{main-theorem} Suppose that $\mu$  is a complex regular Borel measure  with compact support $\Gamma \subset \S'$ 
 such that 
$\mu\left(\Gamma \right) \not = 0$   and let
$F: \Gamma \to \C$ be a continuous function  such that 
 $F\mu \not =0$. 
Then the following are equivalent
\begin{enumerate}
\item  The measure  $\mu$ satisfies the covariance equation with respect to $F$.
\item The measure  $|\mu|$ satisfies the covariance equation with respect to $F$.
\item   There are a constant $c\in \C\setminus \{0\}$  and an element $\gamma \in \Gamma$ such that $F(\gamma) \not = 0$ and 
\begin{eqnarray}\label{mu} \mu := c\delta_\gamma,\end{eqnarray}
where $\delta_\gamma$ is the Dirac measure at $\gamma.$
\end{enumerate}
\end{mthm}
Indeed, the constant $c$ and the multiplicative function $\gamma$ will be given explicitly in terms of $F$ and $\mu$ by the equalities

 \begin{eqnarray} \label{cgamma} & c  := \mu(\Gamma) 
   \ \ \text{and} \ \  \gamma (s)  := \frac{\widehat{ F\mu}(s, e)}{ \widehat{ F\mu}(e, e)}, s \in \S.
\end{eqnarray}
  
 As a corollary of this,  we obtain the following:
 \begin{corollary}\label{main-corollary} Let $\mu$ be a compactly supported complex regular  Borel measure  on $\S'$  such that
$\mu\left(\S' \right) \not = 0$, 
and let $$\gamma(s):= \frac{\widehat{ \mu}(s, e)}{ \mu\left(\S' \right)}, s \in \S.$$ Then $\mu$ satisfies the equality \begin{eqnarray*} \mu(\S') \int_{\S'} \varrho(s) \overline{\varrho(t)} d\mu(\varrho) = \int_{\S'} \varrho(s) d\mu(\varrho) \int_{\S'}  \overline{\varrho(t) }d\mu(\varrho)  \end{eqnarray*}
for all $s, t \in \S$  if and only if   $\gamma \in \S'$  and  $\mu = \mu(\S') \delta_\gamma$. 
\end{corollary}
As a consequence of the latter result we shall derive the following:
\begin{corollary}\label{thmB}  Let $Y$ be  an integrable complexe random variable such that the expectation $\E[Y]$ is in $\C \setminus\{0\}.$ 
Then a  bounded  complex random vector $X=(X_1, \cdots, X_d)$  satisfies the equation 
\begin{eqnarray}\label{xmn}
 \E[ Y] \E[X^m \overline{X}^nY]  = \E[X^m Y] \E[\overline{X}^n Y] 
\end{eqnarray}
for all multi-indices $m, n \in \N_{0}^d$ if and only $X$  is constant almost surely.
Here,  it is understood that
$$X^m := X_{1}^{m_1} \cdots X_{d}^{m_d}$$
for all  $m = (m_1, \cdots, m_d) \in \N_{0}^d$  
\end{corollary}
The semigroup behind the latter corollary is $\S = (\N_{0}^d, +)$ introduced in Section 5. Variants of this corollary can be stated and proved for other semigroups, but we will omit to do it.

\section{Maserick's approach to $BV$-functions }

Throughout the sequel, $(\S, e,  \ast)$ will  denote a  multiplicative commutative  semigroup  with
identity $e$  and involution $\ast$ and consider the associated product semigroup $\Sde$ and furnish it with the involution denoted again by $\ast$ and  given by
$$(s, t)^\ast := (t^\ast, s^{\ast}), \ \ \text{for all} \ (s, t) \in \Sde.$$
 For each $(a, b) \in  {\Sde},$ we define
  the shift operator $E_{(a,b)} $  
by $$(E_{(a, b)} f)(s, t) = f(as, bt) $$ for all $ (s, t) \in {\Sde} $ and $ f
\in \ \C^{\Sde}. $ The complex span $ {\A}$
of all such operators is  a commutative algebra with
identity  $E_{(e, e)} = I$  and involution
$$\left(\sum_{(a,b)}\alpha_{(a, b)} E_{(a, b)}\right)^{\ast} =
 \sum_{(a, b)}\overline{\alpha_{(a, b)} } E_{(b, a)}$$
 $(\alpha_{(a, b)} \in \ \C,  (a, b)  \in
{\Sde}). $
The semigroup ${\Sde}$
is  embedded in  ${\A}$ as a Hamel basis.  The algebra constructed in this way
is isomorphic to the ${\rm L}^{1}$-algebra (or semigroup algebra) constructed by Hewitt-Zuckermann [9]. The  function space $\ \C^{\Sde}$ can be 
algebraically
identified with 
the dual ${\A}^{\star}$
of ${\A}$ via $T \rightarrow (Tf)(e, e) $  $ (f
\in \ \C^{\Sde} $ and $T \in {\A}).$ 
This map topologically identifies $ \C^{\Sde} $ equipped with the topology of simple convergence and ${\A}^{\star}$ equipped with the weak*-topology.
Following Berg et al. [2],  ${(\Sde)}^{\ast}$ we will denote the set of all  semicharacters on ${\Sde}. $ That is, ${(\Sde)}^{\ast}$ consists of   the set of all members $\eta$ of $\ \C^{\Sde}$ such that $\eta(e, e) = 1 $ and 
$$\eta(s s', tt')^{\ast}) = \eta(s, t)\overline{\eta(s', t')} \ \ \text{ for all} \  (s, t), (s', t') \in {\Sde}. $$  
 Equipped with the   topology of pointwise convergence, ${(\Sde)}^{\ast}$ is  completely regular.  It is clear that if $\varrho \in \S',$ then
 $$\eta_\varrho(s, t) :=  \varrho(s) \overline{\varrho(t)} , \ (s, t) \in \Sde$$
 is a semicharacter on $\Sde$. Moreover, the mapping $\varrho \mapsto \eta_\varrho$ is a homeomorphism from $\S'$ onto $(\Sde)^\ast.$ This identifies Radon measures on 
  $\S'$ with those on $(\Sde)^\ast.$ 
  
  \begin{definition}  Let  $\tau$ be a subset of $\A$ and let $\Ptau$ be convex cone in $\A$ spanned by positive linear sums of finite products   of elements of $\tau.$   The subset $\tau$ is called admissible the sense of Maserick if \begin{enumerate}
  \item $T^\ast = T$ for each $T\in \tau.$
  \item $I - T\in \Ptau$ for each $T\in \tau.$
  \item $\A$ is spanned by linear sums of finite products   of members of $\tau.$
  \end{enumerate}
 
  \end{definition}
  
 An example of an admissible set is $ \tau := \{ T_{a, \sigma}, a, \in \Sde, \sigma\in\{-1, 1, -i, i\}\}$ where
 \begin{eqnarray}\label{tau} T_{a, \sigma} := \frac{1}{4}\left(I + \frac{\sigma}{2} E_a + \frac{\bar \sigma}{2} E_{a^\ast}\right).\end{eqnarray}
 
  Consider a function $f :\Sde \to \C$  and let $\Omega$  be a collection of subsets of $\Ptau$ that satisfies 
 $$\sum_{T\in \Lambda}  T = I, \ \ \text{for all} \ \ \Lambda \in \Omega,$$
 
$$ \{TT': (T, T') \in \Lambda \times \Lambda \} \in \Omega \ \ \text{for all} \ \ \Lambda \in \Omega$$
and  each element of $\tau$ belongs to some  $\Lambda \in  \Ptau.$ We impose the ordering on $\Ptau$ given by
$$\Lambda_1 \leq \Lambda_2 \Longleftrightarrow  \Lambda_2 =  \Lambda_3 \Lambda_2$$ 
for some $ \Lambda_3 \in \Ptau$.  Then   define 
$$\|f\|_\Lambda := \sum_{T\in \Lambda}  |Tf(e,e)|.$$
It is clear that the function $\Lambda \mapsto \|f\|_\Lambda$ is nondecreasing.  The function $f$ is said to be of bounded variation  (or just a $BV$-function ) in the sense of Maserick if 
$$\|f \| := \lim_{\Lambda} \|f\|_\Lambda < +\infty.$$
We point out that this definition was given by Marserick for arbitrary commutative semigroups with involution.  Moreover, 
  in a series of articles  \cite{M1}, \cite{M2}, \cite{M3}, he proved that $BV$-functions are precisely
  generalized Laplace transforms of compactly supported complex regular Borel measures on the the space of semi-characters.

  To give another way of describing generalized Laplace transforms we recall \cite{BCR} that   a function $f : \Sde \to \C$ is said to be positive definite if 
  $$\sum_{j,k=1}^n c_j \overline{c_k} f(a_j + a_{k}^{\ast}) \ \geq 0$$
  for $n \in \N $ and finite sequences $(c_j) \in \C^n,$ and $(a_j)  \in  (\Sde)^n.$
  It was shown  \cite{M1}, \cite{M2}, \cite{M3} that the space of bounded  $BV$-functions is precisely the complex span of bounded 
  positive definite functions. More generally, it was proved there that $BV$-functions complex linear sums of  
  positive definite functions which are bounded in a certain sense (exponential boundedness).
  
  Finally, using the above,  Corollary 1 can be rephrased as 
   \begin{cor1prime} Let $f: \Sde \to \C$ be a  $BV$-function  in the sense of Maserick  such that $f(e, e) \not =0.$ Then $f$ satisfies the equality 
   $$f(e,e) f(s, t) = f(s, e) f(e, t), \ \ \ \ \ \text{for all} \ \ (s, t) \in \Sde$$
   if and only if there are a constant $c\in \C\setminus\{0\}$ and a multiplicative function $\gamma \in \S'$ such that
   $$f(s,t) = c \gamma(s) \overline {\gamma(t)},    \ \ \text{for all} \ \ (s, t) \in \Sde.$$
\end{cor1prime}
  We point out that in the case of the semigroup $\S := (]0, 1], \wedge)$ where $s\wedge t := \min(s, t)$,  with the  choice of $\tau$ as in (\ref{tau}) then  $BV$-function  on $]0, 1]\times ]0, 1]$  in the sense of Maserick coincide with $BV$-functions in the classical sense. This can be deduced  from \cite{M2}.
   
   \section{Finite rank Toeplitz operators on the unit disc }  
   
  We recall the basic  scheme of Luecking's method \cite{L} about finite rank Toeplitz operators. Let $ \D := \{ z \in \C : |z| < 1\}$  be the unit disc of $\C$ and let $A^2(\D)$ be classical Bergman space of $\D$. This is the space of all analytic functions on $\D$ which are square integrable with respect to the Lebesgue measure $dA$ on $\D.$ The corresponding Bergman kernel is given by
  $$K_{\D}(z, w) := \frac{1}{\pi(1- z\bar w)^2}, z, w \in \D.$$
  Given a finite  complex Borel measure
$\nu$  on   $\D$,   the Toeplitz operator with symbol $\nu$ is defined  on the space $\C[z]$ of all analytic polynomials by 
 $$T_\nu f (z) := \int_{\D}   f(w) K_{\D}(z, w)d\nu(w), z \in \D, \ f \in \C[z].$$
 The measure $\mu$ induces  a sesquilinear form defined  $\C[z] \times \C[z]$,  by
$$B_\mu(P, Q) := \int_{\C} P(z) \overline{Q(z)} d\nu(z)$$
for all $P, Q \in \C[z].$ 
The reproducing formula for the Bergman kernel shows that 
$$B_\mu(P, Q) = \int_{\D}   (T_\nu P)(z) \bar Q(z) dA(z)$$
for all analytic polynomials $P$ and $Q$.
 We are interested in those bounded Toeplitz operator which   on  $A^2(\D)$ which are of finite rank. The characterization of those measures which induce such operators is given  by
   the following result due to Luecking \cite{L}.
\begin{Leucking}\label{rankC}  Suppose that $ \nu $ is a  finite complex Borel measure on $\D$. Then for following are equivalent: \begin{enumerate}
\item  $T_\nu$ has finite rank
 \item $\nu$ has finite support.
 \end{enumerate}
 Moreover, in the affirmative case,  the rank of $T_\nu$ is equal to the number of points in the support of $\nu$.
\end{Leucking}
The latter theorem will be of particular use in the proof of our main theorem.
 \section{Proof of the main results }  

  Throughout this section,  let $\Gamma$ be  a compact subset of  $\S'$ and let $\mu$ be complex Borel measure on $\Gamma$ such that $\mu(\Gamma) =1$.    Then for each    $s\in \S$ the number
  $$|s|_{\Gamma} := \sup_{\varrho \in \Gamma} | \varrho(s)|$$
   is well-defined and finite due to the continuity of the function $\varrho \to \varrho(s).$

   Finally, we point out that if $ \mu$ satisfies the covariance equation with respect to $F$, then   in view of (\ref{main-equality}) we have that
   
  \begin{eqnarray}\label{pp} \int_{ \Gamma }  P(\varrho(s)) \overline{Q(\varrho(t))} |F(\varrho)|^2 d\mu(\varrho) & = \int_{ \Gamma }  P(\varrho(s)) F(\varrho) d\mu(\varrho)\int_{ \Gamma }   \overline{Q(\varrho(t))} \overline{F(\varrho)} d\mu(\varrho)
 \end{eqnarray}
 for all $s, t \in \S$ and all   analytic polynomials $P, Q.$

  Now, fix $s \in \S$ and consider the complex Borel measure $\nu_s$ on $\D$ defined by 
   \begin{eqnarray}\label{nus} \int_{\D} f(z) d\nu_s(z) &:=  \int_{\Gamma} f\left(\frac{\varrho(s)}{2(1+ |s|_\Gamma)}\right) |F(\varrho)|^2 d\mu(\varrho) 
       \end{eqnarray}
  for all compactly supported continuous functions $f$ on $\D.$
 
   Its obvious that the measure $\nu_s$ is supported in the closed disc of $\C$ centered at the origin with radius $1/2.$
 
 \begin{lemma}\label{Toeplitz}  If  $ \mu$ satisfies the covariance equation with respect to $F$, then the Toeplitz operator $T_{\nu_{s}}$ induced by  $\nu_s$ is of rank at most $1.$
 \end{lemma} 
    \begin{proof} Due to (\ref{pp}),  it follows that for all  $P, Q \in \C[z]$ 
$$ \aligned   \int_{ \Gamma }  P(\varrho(s))  |F(\varrho)|^2 d\mu(\varrho) & = \int_{ \Gamma }  P(\varrho(s)) F(\varrho) d\mu(\varrho)\int_{ \Gamma }  \overline{F(\varrho)} d\mu(\varrho)\\
  \int_{ \Gamma }  \overline{Q(\varrho(s))} |F(\varrho)|^2 d\mu(\varrho) & = \int_{ \Gamma }  F(\varrho) d\mu(\varrho)\int_{ \Gamma }   \overline{Q(\varrho(s)} \overline{F(\varrho)} d\mu(\varrho).
 \endaligned
 $$
 Using this, a little computing shows that for all bounded analytic functions $f$ on $\D$ we have
 $$\aligned   (T_{\nu_{s}} f)(z) & =  \int_{\D}   f(w) K_{\D}(z, w)d\nu_{s}(w) \\   & = 
 \frac{1}{\pi}\sum_{j=0}^{+\infty}  (j+1) z^{j} \int_{ \D}  f(w) \overline{w}^{j} d\nu_s(w)\\
 & =  \frac{1}{\pi}\sum_{j=0}^{+\infty}(j+1) z^{j}\int_{\Gamma} f\left(\frac{\varrho(s)}{2+ |s|_\Gamma}\right) \left(\frac{\overline{\varrho(s)}}{2(1+ |s|_\Gamma)}\right)^{j}|F(\varrho)|^2 d\mu(\varrho) \\
 & =  \frac{1}{\pi}\int_{\Gamma} f\left(\frac{\varrho(s)}{2(1+ |s|_\Gamma)}\right) F(\varrho) d\mu(\varrho) \sum_{j=0}^{+\infty}  (j+1) \int_{\Gamma}  \left(\frac{z\overline{\varrho(s)}}{2(1+ |s|_\Gamma)}\right)^{j}\overline{F(\varrho)} d\mu(\varrho) \\
 & =  \int_{\Gamma} f\left(\frac{\varrho(s)}{2(1+ |s|_\Gamma)}\right) F(\varrho) d\mu(\varrho)  \int_{\Gamma}  K_\D\left(z, \frac{\varrho(s)}{2(1+ |s|_\Gamma)}\right) \overline{F(\varrho)} d\mu(\varrho). 
 \endaligned 
 $$
 Hence $T_{\nu_{s}} f$ belongs to the linear span of the element
 $$ z \mapsto \int_{\Gamma}  K_\D\left(z, \frac{\varrho(s)}{2(1+ |s|_\Gamma)}\right) \overline{F(\varrho)} d\mu(\varrho)$$
 of $A^2(\D).$ 
In particular, if $\nu_s(\D)\not = 0,$ then $$\aligned   (T_{\nu_{s}} f)(z)  & =  \int_{\Gamma} f\left(\frac{\varrho(s)}{2(1+ |s|_\Gamma)}\right) F(\varrho) d\mu(\varrho)  \int_{\Gamma}  K_\D\left(z, \frac{\varrho(s)}{2(1+ |s|_\Gamma)}\right) \overline{F(\varrho)} d\mu(\varrho) 
 \\
 & =  \frac{\int_{\Gamma} f\left(\frac{\varrho(s)}{2(1+ |s|_\Gamma)}\right) F(\varrho) d\mu(\varrho)}{\int_{\Gamma}  F(\varrho) d\mu(\varrho)}\int_{\D}    K_{\D}(z, w)d\nu_{s}(w)\\
 & =  \frac{\int_{\D} f(w)d\nu_s(w)}{ \nu_s(\D)} \int_{\D}    K_{\D}(z, w)d\nu_{s}(w)
 \endaligned 
 $$
This completes the proof.
     
    \end{proof}

 \begin{lemma}\label{one-rank}  Suppose that $ \mu$ satisfies the covariance equation with respect to $F$.
 Then the 
 
 Then the following are equivalent.
 \begin{enumerate} 
 \item  For all $s \in \S$, the Toeplitz operator $T_{\nu_{s}}$ induced by  $\nu_s$ is of rank  $1$.
 \item  There exists $s \in \S$ such that Toeplitz operator $T_{\nu_{s}}$ induced by  $\nu_s$ is of rank  $1$.   
  \item $F\mu \not = 0.$ 
  \item $ \int_{\Gamma}  F(\varrho) d\mu(\varrho)  \not = 0.
      $
       \item $ \int_{\Gamma}  \overline{F(\varrho)} d\mu(\varrho)  \not = 0.
      $
      \item $ \int_{\Gamma}  F(\varrho) d\mu(\varrho)  \not = 0$ and $
    \int_{\Gamma}   \overline{F(\varrho)}  d\mu(\varrho)  \not = 0. 
      $ 
  \item  $
    \int_{\Gamma}  |F(\varrho)|^2 d\mu(\varrho)  \not = 0. 
      $ 
     \item  For all $s \in \S$,  we have $
    \int_{\Gamma}  |F(\varrho)|^2 d\mu(\varrho)  \not = 0, 
      $
     and there is a unique complex number $a_s \in  \D $ such that
      \begin{eqnarray}\label{diracs} \nu_s &:=   \int_{\Gamma}  |F(\varrho)|^2 d\mu(\varrho)   \delta_{a_s}
       \end{eqnarray}
       where $ \delta_{a_s}$ is the Dirac mass at $a_s.$
              \end{enumerate} 
   \end{lemma}
 \begin{proof}   The facts that  $(1) \Rightarrow (2)$, $ (6)  \Rightarrow (4) \Rightarrow(3)$,  
 $ (6)  \Rightarrow (5) \Rightarrow(3)$ and $(8) \Rightarrow (7)$ are obvious.  That $(7) \Rightarrow (6)$ follows from (\ref{pp}) The proof of Lemma \ref{Toeplitz} shows that 
 $(2) \Rightarrow (3)$.  
  Finally, assume that $(1)$ holds. By Luecking's theorem there are complex numbers $m_s \in \C\setminus\{0\}$ and  $a_s \in  \D $ such that
 $$\nu_s =   m_s  \delta_{a_s}.$$
 Therefore,
 $$\int_{\Gamma}  |F(\varrho)|^2 d\mu(\varrho) = \nu_s(\D) = m_s \not =0$$
 showing that $(8)$ holds.
This completes the proof of the lemma.
  \end{proof} 
    We are ready now to prove the Main Theorem.
 \begin{proof-mthm}  {\rm First we prove that $(1)  \Rightarrow (3)$. Assume that $\mu$ and $F$ satisfy the hypothesis  $(1)$ of the Main Theorem. Without loss of generality we may  assume that    $\mu(\Gamma)  = 1.$  If   $ \mu$ satisfies the covariance equation with respect to $F$, then by  Lemma \ref{one-rank} we see that 
 $$  \int_{\D}  f(z) d\nu_s(z) = f(a_s) \int_{\Gamma}  |F(\varrho)|^2 d\mu(\varrho),  $$
 for all  $s\in \S$ and continuous functions on  $\D)$.  
 Setting
 $$\gamma(s) := 2(1+|s|_\Gamma)a_s, \ \ s \in \S,$$
 this yields
 $$  \int_{\Gamma}  f(\varrho(s)) |F|^2(\varrho) d\mu(\varrho) = f(\gamma(s)) \int_{\Gamma}  |F(\varrho)|^2 d\mu(\varrho),  $$
 for all  $s\in \S$ and continuous functions on  $\C$.  Taking $f(z) =   z$  and applying (\ref{pp}) yields
 $$\aligned 
  \gamma(s)  & = \frac{\int_{\Gamma}  \varrho(s) |F|^2(\varrho) d\mu(\varrho)}{ \int_{\Gamma}  |F(\varrho)|^2 d\mu(\varrho)}\\
  & = \frac{\int_{\Gamma}  \varrho(s) F(\varrho) d\mu(z)}{ \int_{\Gamma}  F(\varrho) d\mu(\varrho)}
 \endaligned
 $$
 for all $s\in \S.$  Taking $f(z) =  \bar z$  and applying (\ref{pp}) yields

 $$\aligned 
 \overline{ \gamma(s) } & = \frac{\int_{\Gamma}  \overline{\varrho(s)} |F|^2(\varrho) d\mu(\varrho)}{ \int_{\Gamma}  |F(\varrho)|^2 d\mu(\varrho)}\\
  & = \frac{\int_{\Gamma} \overline{\varrho(s)}  F(\varrho) d\mu(z)}{ \int_{\Gamma}  \overline{F(\varrho)} d\mu(\varrho)}
 \endaligned
 $$
 for all $s\in \S.$  Using this and applying again (\ref{pp}) a little computing shows that for all polynomials $P, Q \in\C[z]$ and pairs $(s, t) \in \S \times \S,$ we have 
 $$  \int_{\Gamma}  P(\varrho(s))  \overline{Q(\varrho(t))}|F|^2(\varrho) d\mu(\varrho) = P(\gamma(s)) \overline{Q(\gamma(t))}\int_{\Gamma}  |F(\varrho)|^2 d\mu(\varrho),  $$
 Now by the Stone-Weierstrass theorem this implies that  
 $$  \int_{\Gamma}  f(\varrho)|F|^2(\varrho) d\mu(\varrho) = f(\gamma)\int_{\Gamma}  |F(\varrho)|^2 d\mu(\varrho),  $$
 for all continuous functions $f$ on $\Gamma \cup\{\gamma\}.$ Hence $\gamma \in \Gamma$ and 
 $$|F|^2\mu = \left(  \int_{\Gamma} |F|^2(\varrho) d\mu(\varrho) \right) \delta_\gamma.$$
 This proves that part $(3)$   of the theorem holds. The proof of the converse if straightforward. Finally, it is obvious  to see that the fact 
 $(3) \Rightarrow (2)  \Rightarrow (1)$ is true,  which completes the proof.}
 
 \end{proof-mthm}

Next we prove Corollary  \ref{main-corollary}. 
\begin{proof}  Follows by  taking the constant function $F\equiv 1$ in the main theorem. This completes the proof.
\end{proof}

The proof of Corollary 2 will be shifted to the next section.
\section{Applications to classical semigroups}
  Although the theory works for general semigroups, we limit ourselves concrete semigroups to illustrate the results. There are many interesting examples of such semigroups for which the multiplicative functions are explicit  \cite{BCR}. 
  
  \subsection{ The semigroup $(\N_{0}^{d}, +)$ }  For a positive integer $d$ the  semi-group $(\N_{0}^{d}, +)$ is the set of all multi-indices $m= (m_1, \cdots, m_d)$, where the entries are nonnegative integers, equipped the standard addition.  The case $d=1$ corresponds to the setting of the question posed in the 
   paper \cite{EKKMR} asking for the description of those signed measures $\mu$ on the closed unit disc $\cD$  that satisfy the equality
\begin{eqnarray}\label{ekkmr} \mu(\cD) \int_{\cD} z^{n} {\bar z}^m d\mu(z) =  \int_{\cD} z^{n} d\mu(z) \int_{\cD}  {\bar z}^m d\mu(z)\end{eqnarray}
for all $m, n$  in the set $\N_{0}$ of all non-negative integers.  A direct application of our main result  solves this problem even in a more general abstract setting. Indeed, 
\begin{proposition}\label{Nd}  Let $d$ be a positive integer. Suppose that $F$ is continuous function on $\C^d$ and  $\mu$ is a compactly supported complex Borel measure on $\C^d$ such that If $\mu(\C^d)  \not = 0 $ and $F\mu  \not = 0, $ then  the following are equivalent

\begin{enumerate}
\item  $\mu$ satisfies the covariance equation with respect to $F$.
\begin{eqnarray}\label{main-equality3}
\mu(\C^d) \int_{\C^d} z^{n} {\bar z}^m |F(z)|^2d\mu(z) =  \int_{\C^d} z^{n} F(z)d\mu(z) \int_{\C^d}  {\bar z}^m \overline{F(z)}d\mu(z)\end{eqnarray}
for all multi-indices $m, n \in \N_{0}^d$
\item There are a constant $c\in \C\setminus \{0\}$ and a vector $ \zeta \in \C^d$ such that $F(\zeta) \not = 0$ such that $\mu = c \delta_\zeta.$
\end{enumerate}
\end{proposition}
\begin{proof}  As mentioned before, we consider the finitely generated abelian semigroup $\S = (\N_{0}^d, +)$ with neutral element $0.$  It is not hard to see the corresponding semigroup $\S'$ can be identified algebraically and topologically  with $\C^d$ following the correspondence 
$$\varrho_z(m) = z^m,   \ \ m \in \N_{0}^d,  \ z \in \C^d$$
with the understanding that $0^0 = 1.$ To complete the proof, it suffices to apply Corollary \ref{main-corollary} to this semigroup.
\end{proof}

\begin{remark}{\rm  Under the assumptions of Proposition \ref{Nd}, if $F\equiv 1$ and $\mu(\C^d)  \not = 0, $ 
  then $\mu$ satisfies (\ref{main-equality3}) if and only if there is  a nonzero constant $c \in \C$ such that $\mu = c\delta_\zeta$ where
  $$\zeta := \frac{1}{\mu(\C^d) } \left(\int_{\C^d}z_1d\mu(z), \cdots, \int_{\C^d}z_d  d\mu(z)\right).$$
  Hence the reduction of this result to the case $d=1$ and  $F\equiv 1$ solves the problem (\ref{ekkmr}) posed by \cite{EKKMR}.}

\end{remark}

Now we prove Corollary 2.
\begin{proofcor2}{\rm   Suppose that  $Y$ is an integrable   random variable with expectation $\E[Y] \not = 0.$ Let $X$ be   a bounded  complex random  vector and  which satisfies (\ref{xmn}).   We may assume that  the complex random vector $X$  takes its values inside  a compact $\Gamma$ of $\C^d$.  Consider the  compactly supported complex Borel $\mu$ on $\C^d$ defined by
\begin{eqnarray}\label{proba1}
\int_{\C^d} f(z)d\mu(z) =  \E[f(X)Y]\end{eqnarray}
for all continuous function on  $\C^d$  supported in $\Gamma$.  Then $\mu(\C^d) = \E[Y] \not = 0$ and $\mu$ satisfies
(\ref{main-equality3}) with $F\equiv 1$ so that by Proposition \ref{Nd} we see that for some vector $\zeta\in \C^d$ and a nonzero constant $c$ we have $\mu = c \delta_\zeta.$ Hence
$$\E[f(X)Y] = c f(\zeta)$$
for all  bounded Borel measurable functions $f$ on $\C^d.$  Denote by $\E[Y|X=x]$ the conditional expectation of $Y$ given $X=z$. Then the latter equality implies that
$$ \int_{\C^d} \E(Y| X = z) f(x) dP_X(z) = c f(\zeta)$$
for all bounded Borel measurable functions $f$ on $\C^d$ showing that 
$$dP_X = \delta_\zeta \ \ \ \text{and} \ \ \E(Y| X = \zeta) = c = \E(Y).$$
The converse is obvious, completing the proof.}
\end{proofcor2}

 \subsection{ The multiplicative 
semi-group $(\N, .)$ } This is the set of all positive integers equipped the natural product with neutral element $1.$ 
Let $(p_{j})_{j\in \N}$ be the sequence of primes. The fundamental theorem of  arithmetic ensures that each  $n\in \N$ admites a unique representation $n = p_{1}^{n_1} \cdots p_{l}^{n_l}, $  where $l \in \N$ and $n_1 \cdots n_l\in \N_0.$ we set $\kappa(n) := (n_1, \cdots,  n_l)$ and for  each sequence $z =(z_j)_{j\in \N},$  of complex numbers, we write 
$$z^{\kappa(n)} := \Pi_{j=1}^l z_{j}^{n_{j}}.$$

 It is not hard to see that the topological semigroup $(\N, .)'$ of multiplicative  functions on  $(\N, .)$  can be identified algebraically and topologically with $ \Castinfty := \C^{\N}$ be furnished with product topology
 following the correspondence 
$$\varrho_z(n) = z^{\kappa(n)},   \ \ n \in \N,  \ z \in \Castinfty.$$
  The reduction of our main result to this semigroup gives the following.
  
   \begin{proposition}\label{Nproduit}  Let $F: \Castinfty \to \C$ be a continuous function and let  $\mu$ be a  compactly supported complex regular Borel measure  on $\Castinfty $  such  that $\mu(\Castinfty)  \not = 0$ and $F\mu \not = 0.$  Then $\mu$ satisfies the covariance equaltion
\begin{eqnarray}\label{main-equalityNproduit}
\mu(\Castinfty) \int_{\Castinfty} z^{\kappa(n)} {\bar z}^{\kappa(m)} |F(z)|^2d\mu(z) =  \int_{\Castinfty} F(z) z^{\kappa(n)} d\mu(z) \int_{\Castinfty}   \overline{ z^{\kappa(m)}F(z)} d\mu(z)\end{eqnarray}
for all multi-indices $m, n \in \N$ if and only if there are a nonzero constant $c$ and an element $\zeta $ of $\Castinfty$ such that $F(\zeta) \not = 0$ and $\mu = c \delta_\zeta.$

\end{proposition}
\begin{remark}{\rm  Under the assumptions of Proposition \ref{Nproduit}, if $F$ is the constant function $1$ and $\mu(\Castinfty)  \not = 0, $ then
  then $\mu$ satisfies (\ref{main-equalityNproduit}) if and only if there is  a constant $c \in \C$ such that $\mu = c\delta_\zeta$ where
  $$\zeta := \frac{1}{\mu(\Castinfty) } \left(\int_{\Castinfty}z_1d\mu(z), \cdots, \int_{\Castinfty}z_d  d\mu(z)\right).$$}
 
\end{remark}

 \subsection{ The semigroup $\S = ([0, +\infty), +)$ } The continuous multiplicative on this semigroup are form $\varrho_z(s) = e^{sz}, z\in \C$.
 Consider the subset  of  $\S' $ given by
 $$\Gamma := \{s \mapsto e^{-sz}, z\in \C,  \ \text{Re } z \geq 0 \} \cup \{\infty\}$$ 
 with the understanding that $\varrho_\infty (s) = 1_{\{0\}}(s), s \geq 0.$ Then $\Gamma$ is a compact subset 
 of  $\S'.$ Furthermore, complex regular Borel measures on 
 $\Gamma \setminus  \{\infty\}$ are in a bijective correspondence with finite complex regular measure on the right halfplane  $\H =  \{z\in \C,  \ \text{Re } z \geq 0 \}.$ 
 Hence every finite complex regular measure $\mu$ on $\H$ induces  a compactly supported finite complex regular Borel measure $\nu$ on $\Gamma$  defined by its generalized Laplace transform
 $$ \widehat \nu (s, t) = \int_{\Gamma \setminus  \{\infty\} } \varrho(s) \overline{\varrho(t)}d\nu(\varrho) =   \int_{\H} e^{-sz - t\bar z} d\mu(z),  \ (s, t) \Sde.$$
 Applying the main theorem to this semigroup yields   the following
   \begin{proposition}\label{H}  Let $F: \H \to \C$ be a continuous function and let  $\mu$ be a finite complex Borel measure  on $\H $  such  that $\mu(\H)  \not = 0$ and $F\mu \not = 0.$  Then $\mu$ satisfies the covariance equation with respect to $F$
\begin{eqnarray}\label{main-equalityH}
\mu(\H) \int_{\H} e^{-sz - t\bar z }|F(z)|^2d\mu(z) =   \int_{\H} e^{-sz} F(z)d\mu(z)  \int_{\H} e^{ - t\bar z} \overline{F(z)}d\mu(z)   \end{eqnarray}
for all  $s, s\in [0, +\infty)$ if and only if there are a nonzero constant $c$ and an element $\zeta $ of $\H$ such that $F(\zeta) \not = 0$ and $\mu = c \delta_\zeta.$
\end{proposition}
\begin{remark}{\rm The semigroup $\S = ([0, +\infty), +)$ is a sample example of the so called
 {\it conelike semigroups}.  These are subsets $\S$ of finite dimensional vector spaces that are stable under addition, contain the origin $0$ 
and satisfy: for all $x \in \S,$ there is a real number $r_x \geq 0$  such that $ r x \in \S $ for all $r\geq r_x.$  They were introduced 
 by Ressel \cite{Re} in connection with Bochner's type theorem for topological semigroups. Using the work \cite{Y1}, these semigroups were shown  to be  of interest to holomorphic functions in tube domains \cite{Y2}.
 Finally,  our main theorem can be applied to derive results in the same spirit as Proposition \ref{H} in the case of these semigroups.}
\end{remark}

\section{Analytic kernels}

This section offers further applications of our results. We will establish a kind of extremal property  for analytic kernels.
For simplicity we consider the ball setting.  Consider and open $\B_d$ in $\C^{d} $ and a closed ball $\overline{\B}_p$ in $\C^{p} $ both centered at $0$ with positive radius. Assume that $K: \B_d \times \overline{ \B}_p \to \C$ be kernel of the form
\begin{eqnarray}\label{K}   K(z, w) = \sum_{m, n \in \N_{0}^d} a_{m, n} z^m  w^n, 
\end{eqnarray}
with nonzero complexe coefficients $a_{m, n}$, where the series is uniformly convergent   in the variable $w$ for all fixed  $z.$  We will prove the following 
\begin{theorem}\label{kernel}  Suppose that $K$ is a kernel of the form (\ref{K}) and let $f:\B_d \to \C$ be holomorphic function. Then a finite  complex Borel measure 
 $\mu$ on $\overline{\B}_p$  such that  $\mu(\overline{\B}_p) = 1$ satisfies  the equality
\begin{eqnarray}\label{kernel-equation1} |f(z)|^2 =   \int_{ \overline{\B}_p}|K(z, \overline{ w})|^{2} d\mu(w),
\end{eqnarray}
in a neighborhood of $0$ if and only if 
 there is are complex constant $c$ and an element $\zeta $ of $\B_p$ such that 
 $$\mu = c \delta_\zeta \ \ \text{and}\ \ f =  cK(., \bar \zeta)$$
\end{theorem}
\begin{proof} We expand $f$ in terms of a power series   $f(z) = \sum_{m\in \N_{0}^d} b_m z^m$ in a neighborhood of $0$.  If $\mu$ satisfies
 (\ref{kernel-equation1}), then by integration, uniqueness of the power series expansion and (\ref{K}) we see that
$$ b_p \overline{b_q} =   a_{m, q} \overline{a_{p, n} }\int_{ \overline{\B}_p} w^n \overline{w}^m d\mu(w),$$
for all $m, n, p, q  \in \N_{0}^d.$ The proof now follows from Proposition \ref{Nd}.
\end{proof} 
\section{concluding remarks}

\begin{remark}{\rm In view of Lemma 4.1 in \cite{Y1} we see that generalized Laplace transforms of nonnegative measure are preserved under surjective morphisms of semigroups. This, combined with Marserick's work \cite{M2},  shows that  generalized Laplace transforms of compactly supported complex regular Borel measures remain preserved under these mappings.  Below, we list a couple examples:

 Since the mapping $x \mapsto e^{x}$ is an isomorphism from  the additive $ \S= ([0, +\infty), +)$ onto the multiplicative  semigroup $ \T =  ([1, +\infty), .)$ The analog of Proposition \ref{H}  on $\T$ can be derived by a simple change of variable. 
 
 Every finitely generated abelian semigroup $\T$ with unit  is surjective image  of  $(\N_{0}^d, +, 0)$ for some $d\in \N,$ hence a similar observation shows an analog of Proposition \ref{Nd}  on $\T$ can be deduced i by change of variable.
}
\end{remark}

\end{document}